\documentclass[12pt]{article}
\usepackage[utf8]{inputenc}
\usepackage{amsmath, amsthm, amssymb}
\usepackage{thmtools, thm-restate}
\usepackage{setspace}
\usepackage{fullpage}
\newtheorem{theorem}{Theorem}

\usepackage{titlesec}
\titlespacing*{\section}{0pt}{0.3\baselineskip}{0.2\baselineskip}

\begin{document}

\author{J\'ozsef Balogh\thanks{Department of Mathematical Sciences, University of Illinois at Urbana-Champaign, Urbana, Illinois 61801, USA, 
and Moscow Institute of Physics and Technology, Russian Federation. Partially supported by NSF
Grant DMS-1500121 and DMS-1764123, Arnold O. Beckman Research Award (UIUC Campus Research Board RB 18132) and
the Langan Scholar Fund (UIUC).}  
\and
William Linz\thanks{Department of Mathematical Sciences, University of Illinois at Urbana-Champaign, Urbana, Illinois 61801, USA. Email: wlinz2@illinois.edu.}
\and
Let\'{\i}cia Mattos\thanks{IMPA, Estrada Dona Castorina 110, Jardim Bot\^anico, Rio de Janeiro, 22460-320, Brazil. Partially supported by CAPES. Email: leticiamat@impa.br.}}
        
\title{Long rainbow arithmetic progressions}        
\date{\today}

\maketitle

\maketitle
\begin{abstract}
Define $T_k$ as the minimal $t\in \mathbb{N}$ for which there is a rainbow arithmetic progression of length $k$ in every equinumerous $t$-coloring of $[tn]$ for all $n\in \mathbb{N}$. Jungi\'{c}, Licht (Fox), Mahdian, Ne\u{s}et\u{r}il and Radoi\u{c}i\'{c} proved that $\lfloor{\frac{k^2}{4}\rfloor}\le T_k\le k(k-1)^2/2$. We almost close the gap between the upper and lower bounds by proving that $T_k \le k^2e^{(\ln\ln k)^2(1+o(1))}$. Conlon, Fox and Sudakov have independently shown a stronger statement that $T_k=O(k^2\log k)$.    
\end{abstract}
\section{Introduction}
An \textit{equinumerous} coloring of any set of objects is a coloring in which each color is used the exact same number of times in the coloring. Given a coloring of $[n]$, an arithmetic progression in $[n]$ is  \textit{rainbow} if each term of the arithmetic progression has a different color. We denote a $k$-term arithmetic progression by the shorthand $k$-AP. 

Jungi\'{c}, Licht (Fox), Mahdian, Ne\u{s}et\u{r}il and Radoi\u{c}i\'{c} \cite{originalrainbow} defined $T_k$ to be the minimal $t$ so that every equinumerous $t$-coloring of $[tn]$ contains a rainbow $k$-AP for every $n \in \mathbb{N}$. They proved the bounds $\lfloor{\frac{k^2}{4}\rfloor}\le T_k\le \frac{k(k-1)^2}{2}$ for every $k\ge 3$, and furthermore they conjectured that $T_k=\Theta(k^2)$. Little is known about exact values of $T_k$: Axenovich and Fon-Der-Flass~\cite{axenovich} and Jungi\'{c} and Radoi\u{c}i\'{c}~\cite{jungict3case} independently proved that $T_3=3$, and this remains the only value of $k$ for which $T_k$ is known exactly. Variations of this problem to understand the \textit{anti-van der Waerden numbers} have been considered by Butler \textit{et al.} \cite{butleretal}. 

Quite recently, Geneson~\cite{genesonimprovement} proved the upper bound $T_k = k^{\frac52+o(1)}$. Geneson~\cite{genesonimprovement} achieved this improvement by making a more careful study of the possible divisors of the differences between consecutive entries in the same color class of a given equinumerous $t$-coloring of $[tn]$, and by utilizing the K\H{o}v\'{a}ri-S\'{o}s-Tur\'{a}n theorem. In this note, we improve the upper bound in \cite{genesonimprovement} to 
almost match the lower bound of \cite{originalrainbow}. 

\begin{theorem}\label{mainthm}
$T_k\le k^2e^{(1+o(1))(\ln\ln k)^2}$, as $k\rightarrow\infty$.  
\end{theorem}
A stronger result, that $T_k=O(k^2\log k)$, was obtained independently by Conlon, Fox and Sudakov \cite{conlonfoxsudakovresult}. Compared to our proof, their method considers fewer $k$-APs, but they are able to obtain a better bound because they overcount each $k$-AP only once. 


\section{Proof of Theorem \ref{mainthm}}

Let $t$ be the minimum number so that there is a rainbow $k$-AP for every equinumerous $t$-coloring of $[tn]$ for every $n\ge 1$. From the bounds of \cite{originalrainbow}, we may assume that $\lfloor{\frac{k^2}{4}\rfloor}\le t < k^3$. Furthermore, the number of $k$-APs in $[tn]$ is greater than 
\[
\dfrac{tn(tn-3(k-1))}{2(k-1)}.
\]

Let $D:= \{d\in [tn /10k, tn/2k]:  \nu_{k}(d) < (1+o(1))\ln\ln k\}$, where $\nu_{k}(d)$ is the number of prime divisors of $d$ that are at most $k$, counted with multiplicity. Here the $o(1)$ term is as $k\rightarrow \infty$. 

\begin{restatable}[]{lemma}{primedivisorlemma}
\label{primedivisorlemma}
$|[tn / 10k, tn / 2k] \setminus D|=o(tn/k)$.
\end{restatable}


\begin{proof}[Proof.]
A simple modification of the proof of  Tur\'{a}n's \cite{turan} that almost all integers at most $n$ has about $\ln\ln n$ prime factors (see, for instance, Alon and Spencer \cite[pp.~45--46]{alonspencer}) shows that the number of integers that are at most $tn / 2k$ and which have more than $(1+o(1))\ln\ln k$ prime divisors at most $k$ is $o(tn/k)$; we omit the details. 
\end{proof}

Let $\mathcal{A}$ be the set of $k$-APs in $[tn]$ with difference in the set $D$. We have that 
\begin{equation}\label{Akaps}
|\mathcal{A}| > \frac{t^2n^2}{11k}
\end{equation}
for $k$ sufficiently large.
We count the number of non-rainbow $k$-APs in $\mathcal{A}$. Each such non-rainbow $k$-AP contains a monochromatic pair $(a, b)$. There are $tn$ choices for $a$, and given a choice of $a$, there are at most $n$ choices for $b$. 

We claim that for any pair $(a,b)$, $a,b \in [tn]$, the number of $k$-APs in $\mathcal{A}$ containing $(a,b)$ is bounded by $k\cdot e^{(1+o(1))(\ln \ln k)^2}$. Indeed, either $b-a$ has a representation of the form $b-a = dm$, with $d \in D$ and $m\le k$, or there is no $k$-APs in $\mathcal{A}$ containing $(a,b)$.  In the former case $b-a$ has at most $\log_2 k + (1+o(1))\ln\ln k$ prime factors at most $k$ (with $\log_2 k$ factors coming from $m$ and $(1+o(1))\ln\ln k$ factors coming from $d$). Therefore, the number of ways to factorize $b-a = dm$ is the number of ways to select at most $(1+o(1)) \ln\ln k$ prime factors among all the $\log_2 k + (1+o(1))\ln\ln k$ prime factors that $b-a$ has. That number is upper bounded  by 
\[\binom{\log_2 k + (1+o(1))\ln\ln k}{\le (1+o(1))\ln\ln k} \le e^{(1+o(1))(\ln \ln k)^2}.\]
Finally, given $m$, there are at most $k$ choices for the positions of $a$ and $b$ in a $k$-AP. This implies that the number of non-rainbow $k$-APs containing both $a$ and $b$ is at most $ke^{(1+o(1))(\ln\ln k)^2}$.

Therefore, there are at most
$(tn)(n)(ke^{(1+o(1))(\ln\ln k)^2})$ non-rainbow $k$-APs in $\mathcal{A}$. Combining this with the bound for the number of $k$-APs in $\mathcal{A}$ from \eqref{Akaps}, an upper bound for $T_k$ is given by the smallest $t$ satisfying 
\[ tn^2ke^{(1+o(1))(\ln\ln k)^2}\le \frac{t^2n^2}{11k}.\]

It suffices to take $t=11k^{2}e^{(\ln\ln k)^2(1+o(1))}=k^{2}e^{(\ln\ln k)^2(1+o(1))}$, where the $o(1)$ in the exponent swallows up the factor $11$, completing the proof.



\section*{Acknowledgments}
The authors thank Ran Ji and Mina Nahvi for introducing this problem to us, and also thank Mina Nahvi for help in preparing an early version of this manuscript. The authors thank Kevin Ford for helpful discussions.  The authors thank the referee for carefully reading the manuscript and for useful comments.

{\setstretch{0.88}

}

\begin{thebibliography}{}
\bibitem{alonspencer}
N. Alon, J. Spencer, \textit{The Probabilistic Method}, 4th edition, Wiley, (2016). 

\bibitem{axenovich}
M. Axenovich, D. Fon-Der-Flaass, On rainbow arithmetic progressions, Electronic Journal of Combinatorics, 11, (2004), R1.

\bibitem{butleretal}
S. Butler, C. Erickson, L. Hogben, K. Hogenson, L. Kramer, R. Kramer, J. Lin, R. Martin, D. Stolee, N. Warnberg and M. Young, Rainbow arithmetic progressions, Journal of Combinatorics, 7, (2016), 595--626.

\bibitem{conlonfoxsudakovresult}
D. Conlon, J. Fox, B. Sudakov, Independent arithmetic progressions, arXiv:1901.05084v1, (2019).

\bibitem{genesonimprovement}
J. Geneson, A note on long rainbow arithmetic progressions, arXiv:1811.07989v1, (2018). 

\bibitem{originalrainbow}
V. Jungi\'{c}, J. Licht (Fox), M. Mahdian, J. Ne\u{s}et\u{r}il, R. Radoi\u{c}i\'{c}, Rainbow arithmetic progressions and anti-Ramsey results, Combinatorics, Probability and Computing - Special Issue on Ramsey Theory, 12, (2003), 599--620.

\bibitem{jungict3case}
V. Jungi\'{c}, R. Radoi\u{c}i\'{c}, Rainbow 3-term arithmetic progressions, Integers, The Electronic Journal of Combinatorial Theory, 3, (2003), A18. 


\bibitem{turan}
P. Tur\'{a}n, On a theorem of Hardy and Ramanujan, J. London Math Soc., 9, (1934), 274--276.
\end{thebibliography}
\end{document}